\documentclass[11pt,reqno]{amsart}
\newcounter{defcounter}
\usepackage{amsfonts}
\usepackage{amsmath}
\usepackage{array}
\usepackage{enumerate}
\usepackage{graphicx}
\usepackage{amssymb}
\usepackage{amsthm}
\usepackage{verbatim,amscd,amsmath,amsthm,amstext,amssymb,amsfonts,latexsym,lscape,rawfonts
}
\usepackage[arrow,matrix,curve]{xy}
\usepackage{xcolor}

\newtheorem{theorem}{Theorem}
\newtheorem{corollary}{Corollary}
\newtheorem{lemma}{Lemma}
\newtheorem{proposition}{Proposition}
\theoremstyle{definition}

\newtheorem{remark}{Remark}

\newtheorem*{example}{Example}

\begin{document}
\date{\today}

\title{Eigenvectors of tensors -- A primer}
\author{Sebastian Walcher}
\address{Lehrstuhl A f\"ur Mathematik, RWTH Aachen\\
D-52056 Aachen, Germany\\
{\tt walcher@matha.rwth-aachen.de}
}

\maketitle

\begin{abstract} 
We give an introduction to the theory and to some applications of eigenvectors of tensors (in other words, invariant one-dimensional subspaces of homogeneous polynomial maps), including a review of some concepts that are useful for their discussion. The intent is to give practitioners an overview of  fundamental notions, results and techniques.
\end{abstract}
\section{Introduction}
The notion of eigenvectors of tensors has gained (or rather regained) relevance in recent years, due to work of Cartwright and Sturmfels \cite{CaSt} and to new fields of application, see e.g. Virga and co-authors \cite{Virga, GaVi, CQV}, Oeding et al.~\cite{OeRoStu} and others.  The objects of interest are one dimensional subspaces of a vector space that are invariant with respect to a homogeneous polynomial map of degree $m>1$. In earlier work (see in particular R\"ohrl \cite{RohrlId, RohrlZ}) the notions of idempotents and nilpotents of $m$-ary algebras were also used, generalizing concepts from the theory of algebras.\\
The present paper is a primer; hence its essential function is to give the reader an introduction to this field for real and complex tensors, including some algebraic and analytic-topological methods. When dealing with tensors we avoid indices wherever possible, and rather focus on multilinear maps.\\ Concerning the existence and number of one dimensional invariant subspaces, the fundamental result from classical algebraic geometry is Bezout's theorem, which settles the problem for complex tensors. For the real case we make use of the Brouwer degree, improving some existence results, and of the Poincar\'e-Hopf theorem for gradients of polynomial maps from $\mathbb R^n$ to $\mathbb R$. We discuss applications to ordinary differential equations with homogeneous polynomial right-hand side, and take a look at gradients of cubic tensors in dimension three, which are relevant in liquid crystal theory; see Gaeta and Virga \cite{GaVi} and Chen, Qi and Virga \cite{CQV}. For the discussion of these gradients, we take an approach that is somewhat different from the ones in the cited works. We classify exceptional cases and open the path to an algorithmic determination of the number of real one dimensional subspaces. As a collateral result (so to speak) we include a proof of Bezout's theorem in the appendix, using a combination of algebraic and analytical arguments.

\section{Notions and notation}
\subsection{Tensors}
For the purpose of the present paper, a {\em tensor} of dimension $n$ and order $m\geq 1$ over a field $\mathbb K$ is a multilinear map
\[
\widehat Q: \mathbb K^n\times \cdots\times\mathbb K^n\to\mathbb K^n,\quad (x^{(1)},\ldots, x^{(m)})\mapsto \widehat Q(x^{(1)},\ldots, x^{(m)}).
\]
We restrict attention to $\mathbb K=\mathbb R$ or $\mathbb C$; in particular these are fields of characteristic zero. 
The case $m=1$ corresponds to linear maps; we will be interested in the case $m>1$. $\widehat Q$ is called {\em symmetric} if any permutation of the entries $x^{(1)},\ldots, x^{(m)}$ leaves the value of $\widehat Q(x^{(1)},\ldots, x^{(m)})$ unchanged.\\
Associated to a tensor $\widehat Q$ is the homogeneous polynomial map
\[
Q:\,\mathbb K^n\to\mathbb K^n,\quad x\mapsto Q(x):=\widehat Q(x,\ldots,x)
\]
of degree $m$. 
In coordinates, with $x=(x_1,\ldots, x_n)^{\rm tr}$ one has a representation
\begin{equation}\label{Qhom}
Q(x)=\left(\sum_{(i_1,\ldots,i_n)} \alpha_{i_1,\ldots,i_n}^j x_1^{i_1}\cdots x_n^{i_n}\right)_{1\leq j\leq n}
\end{equation}
with the summation extending over all tuples $(i_1,\ldots,i_n)$ of nonnegative integers that add up to $m$. The {\em structure coefficients} $ \alpha_{i_1,\ldots,i_n}^j $ uniquely determine, and are uniquely determined by $Q$. Thus one may identify the vector space of homogeneous polynomial maps of degree $m$ with the space of structure coefficients, of dimension
\[
n\cdot\begin{pmatrix}n+m-1\\n-1\end{pmatrix}.
\]
Conversely, a homogeneous polynomial map $P$ of degree $m$ defines a symmetric order $m$ tensor $\widetilde P$ through the $m^{\rm th}$ derivative of $P$ at an arbitrary point $a$, thus
\[
\widetilde P(x^{(1)},\ldots, x^{(m)}):= \frac1{m!}D^mP(a)\,(x^{(1)},\ldots, x^{(m)}).
\]
From Euler's identity one sees $\widetilde P(x,\ldots,x)=P(x)$. To summarize, one may identify symmetric tensors and homogeneous polynomial maps.
\subsection{Eigenvectors}
Generalizing the definition for linear maps, one says that $v\not=0$ is an {\em eigenvector} of a tensor $\widehat Q$ or (more appropriately) of the associated homogeneous polynomial map $Q$ if there exists $\lambda\in\mathbb K$ such that
\begin{equation}\label{eveceq}
Q(v)=\lambda v.
\end{equation}
Since \eqref{eveceq} implies
\[
Q(\alpha v) =(\alpha^{m-1}\lambda)\cdot(\alpha v)\quad  \text{for all }\alpha\not=0,
\]
the notion of one dimensional eigenspace is well-defined but the notion of eigenvalue is ambiguous for $m>1$ unless $\lambda=0$. On the one hand we may use this ambiguity to normalize $\lambda\in\{0,\,1\}$ when $\mathbb K=\mathbb C$ or when $\mathbb K=\mathbb R$ and $m$ is even; in case $\mathbb K=\mathbb R$ with odd $m$ we may normalize $\lambda\in\{0,\,1,\,-1\}$. Alternatively, when $\mathbb K=\mathbb R$, normalizing eigenvectors by the requirement $\left<v,v\right>=1$ yields a consistent notion of eigenvalue; this is the point of view taken in Qi \cite{Qi05,Qi07}, Cartwright and Sturmfels \cite{CaSt}, and Gaeta and Virga \cite{GaVi}.\\
In the case $\lambda=0$ we call $v$ a {\em nilpotent} of $Q$; in the case $\lambda=1$ we speak of an {\em idempotent}.
\subsection{Critical points of homogeneous polynomials}\label{gradientstuff}
In some applications a homogeneous polynomial map of degree $m$ appears as the gradient of a homogeneous scalar-valued polynomial $q:\mathbb K^n\to \mathbb K$ of degree $m+1$. We formalize this notion, denoting by $\left<\cdot,\cdot\right>$ the standard symmetric bilinear form on $\mathbb K^n$. For given $q$ there is a unique homogeneous $Q:\,\mathbb K^n\to\mathbb K^n$, of degree $m$, such that the identity
\begin{equation}\label{graddef}
Dq(x)y=(m+1)\left<Q(x),y\right>
\end{equation}
holds for all $x,y\in\mathbb K^n$. An additional property is shown by further differentiation: Since
\[
D^2q(x)\,(y,z)=\left<DQ(x)z,y\right>
\]
and the second derivative $D^2q(x)$ is a symmetric bilinear form for each $x$, the identity
\begin{equation}\label{symmder}
\left<DQ(x)z,y\right>=\left<DQ(x)y,z\right>
\end{equation}
follows; hence $DQ(x)$ is a symmetric linear map for all $x$. Conversely, consider a homogeneous polynomial map $P$ such that $DP(x)$ is symmetric for all $x$. Then one verifies that the assignment
\[
p(x):=\frac{1}{m+1}\left<P(x),x\right>
\]
yields a $\mathbb K$-valued polynomial such that $Dp(x)y=\left<P(x),y\right>$ for all $x$ and $y$. To summarize, there is a $1-1$ correspondence between forms of degree $m+1$ and homogeneous polynomial maps of degree $m$ with symmetric derivative. Moreover, $p$ is harmonic (i.e., the Laplacian of $p$ vanishes) if and only of the trace of $DP(x)$ vanishes for all $x$.\\
\section{Eigenvectors: Algebraic methods}
\subsection{Bezout in projective space}
The fundamental algebraic tool for the discussion of eigenvectors of tensors is Bezout's theorem in projective space. We briefly recall some definitions and facts (see Shafarevich \cite{Shafa}, in particular Ch.~IV, for a full account).
\begin{enumerate}[1.]
\item Projective space $\mathbb P^n$ over $\mathbb K$ is defined as the set of all equivalence classes of nonzero $(n+1)$-tuples 
$(x_0:x_1:\cdots :x_n)$,
with 
\[
(x_0:x_1:\cdots :x_n)= (y_0:y_1:\cdots :y_n)
\]
if and only if there is a $\beta\in\mathbb K^*$ such that $y_i=\beta x_i$ for all $i$.
\item $\mathbb P^n$ is covered by the affine spaces
\[
\mathbb A_i:=\left\{(x_0:x_1:\cdots :x_n); \, x_i\not=0\right\}
\]
which are identified with $\mathbb K^n$ via the bijections
\[
(x_0:x_1:\cdots :x_n)\mapsto (\frac{x_0}{x_i},\ldots,\frac{x_{i-1}}{x_i},\frac{x_{i+1}}{x_i},\ldots,\frac{x_n}{x_i})
\]
\item Let $f_1,\ldots,f_n$ be homogeneous scalar-valued polynomials of degree $m\geq 1$ in $n+1$ variables $x_0,\ldots, x_n$. We may write
\[
f_j(x_0,\ldots, x_n)=\sum_{(i_0,\ldots,i_n)} \beta_{i_0,\ldots,i_n}^j x_0^{i_0}\cdots x_n^{i_n},
\]
summation extending over all tuples of nonnegative integers which add up to $m$. 
\item The notion of a common zero $v= (v_0:\cdots:v_n)\in \mathbb P^n$ of the $f_i$ is unambiguous. The multiplicity of such a zero can be defined in the following way: There is a $k$ such that $v$ is contained in $\mathbb A_k$, say $v\in \mathbb A_0$ for ease of notation. Dehomogenize; i.e. set
\[
\widehat f_j(x)=\widehat f_j(x_1, \ldots, x_n):=f_j(1:x_1:\cdots:x_n), \quad x\in \mathbb K^n.
\]
With $z_i:=v_i/v_0$ for $1\leq i\leq n$, one has that $\widehat f_j(z)=0$, $1\leq j\leq n$. The multiplicity of $z$ (and of $v$) is defined as
\[
\dim\left(\mathcal O_z/\left<\widehat f_1,\ldots,\widehat f_n\right>\right),
\]
with $\mathcal O_z$ the local ring of $z$, consisting of all rational functions on $\mathbb K^n$ that are defined at $z$, and $\left<\widehat f_1,\ldots,\widehat f_n\right>$ the ideal generated by the $\widehat f_j$ in $\mathcal O_z$. 
\item Algorithmic matters: Multiplicities may be computed via standard bases; see e.g. Decker and Lossen \cite{DeLo}.
\item Useful observation: The multiplicity of  $z$ is equal to one if and only if the Jacobian of $\widehat f_1,\ldots, \widehat f_n$ is invertible at $z$.
\end{enumerate}

We come to Bezout's theorem on projective space. For a proof see Shafarevich \cite{Shafa}, Ch.~IV, \S2; another proof is sketched in the Appendix.
\begin{theorem}\label{bezout}
Let $\mathbb K=\mathbb C$ and let $f_1,\ldots, f_n$ be homogeneous polynomials of degree $m$ in $n+1$ variables $x_0,\ldots, x_n$. Then the number of common zeros of the $f_j$ in $\mathbb P^n$ is either infinite or equal to $m^n$, counting multiplicities.
\end{theorem}
\subsection{Resultants, Bezout and eigenvectors}
The application of Bezout's theorem to eigenvectors of tensors is classical; see R\"ohrl \cite{RohrlZ}, Cartwright and Sturmfels \cite{CaSt}, Qi \cite{Qi07}.
Let a homogeneous polynomial map $Q$ from $\mathbb K^n$ to $\mathbb K^n$ be given, keeping the notation from \eqref{Qhom}. We start with an auxiliary result.
\begin{proposition}\label{resultant}
Let $\mathbb K=\mathbb C$, and $Q$ homogeneous of degree $m>1$.
\begin{enumerate}[(a)]
\item Let $X_{i_1,\ldots,i_n}^j $  be indeterminates, with $1\leq j\leq n$ and  $(i_1,\ldots,i_n)$ running through all tuples of nonnegative integers with sum $m$. There is a polynomial $R$ (called the {\em resultant} of the system) in these indeterminates with the property that
\[
R((\alpha_{i_1,\ldots,i_n}^j))=0 \text{ iff }  Q\text { admits a nilpotent.} 
\]
\item Whenever $Q$ admits infinitely many pairwise linearly independent eigenvectors then $Q$ admits a nilpotent.
\end{enumerate}
\end{proposition}
\begin{proof} For properties of resultants and the proof of part (a) see Cox, Little, O'Shea \cite{CLSuse}, Ch.~3. For part (b) consider 
\[
F(x_0,\ldots,x_n):=Q(x_1,\ldots, x_n)-x_0^{m-1}\begin{pmatrix}x_1\\ \vdots \\ x_n\end{pmatrix}
\]
and note that $Q(v)\in\mathbb C v$ if and only if $F(v_0,v_1,\ldots, v_n)=0$ for some $v_0$. The zeros of $F$ form a projective variety, and this variety has dimension $>0$ whenever there exist infinitely many pairwise linearly independent eigenvectors of $Q$. But in $\mathbb P^n$ any variety of dimension $>0$ has nontrivial intersection with the hyperplane defined by $x_0=0$ (see Shafarevich \cite{Shafa}, Ch.~I, \S 6, Thm.~4).
\end{proof}
\begin{proposition}\label{bezcor}
Let $Q$ be a homogeneous polynomial map from $\mathbb C^n$ to $\mathbb C^n$, of degree $m>1$. Then the number of pairwise linearly independent eigenvectors of $Q$ is either infinite or equal to
\[
\frac{m^n-1}{m-1}=\sum_{k=0}^{n-1}m^k,
\]
when multiplicities are counted.
\end{proposition}
\begin{proof} As in the proof of Proposition \ref{resultant}, set 
\[
F(x_0,\ldots,x_n)=Q(x_1,\ldots, x_n)-x_0^{m-1}\begin{pmatrix}x_1\\ \vdots \\ x_n\end{pmatrix}.
\]
If the number of zeros of $F$ in $\mathbb P^n$ is finite then it is equal to $m^n$ by Bezout. 
\begin{enumerate}[(i)]
\item We first consider the case $Q(z)\not=0$ for all $z\not=0$; thus we may normalize all eigenvalues to $1$. The zero $v=(1:0:\cdots:0) $ of $F$ is the only one which does not correspond to an eigenvector of $Q$. For
\[
\widehat F(x_1,\ldots,x_n)=F(1,x_1,\ldots, x_n)=Q(x_1,\ldots, x_n)-\begin{pmatrix}x_1\\ \vdots \\ x_n\end{pmatrix}
\]
the Jacobian $D\widehat F(0,\ldots,0)=-I_n$ is invertible; hence $v$ has multiplicity one, and there are $m^n-1$ remaining solutions, counting multiplicities. Given any $(m-1)^{\rm th}$ root of unity $\zeta$, $Q(c)=c$ implies $Q(\zeta c)=\zeta c$; and conversely $Q(\alpha c)=\alpha^{m-1}\cdot \alpha c$ shows that a scalar multiple of $c$ is an idempotent only if $\alpha^{m-1}=1$. Hence the remaining solutions come in packages of $m-1$ pairwise linearly dependent ones. We thus obtain the asserted formula. 
\item If $F$ has finitely many zeros in $\mathbb P^n$ but $Q(w)=0$ for some $w\not=0$ then let $r$ be an $(m-1)$-form and consider
\[
Q^*(x_1,\ldots,x_n):=Q(x_1,\ldots,x_n)+r(x_1,\ldots, x_n)\cdot \begin{pmatrix}x_1\\ \vdots \\ x_n\end{pmatrix}.
\]
From $Q(c)=\lambda c$ one finds that $Q^*(c)=(\lambda+r(c))\,c$. Therefore $Q$ and $Q^*$ have the same eigenvectors, and by a suitable (generic) choice of $r$ one ensures that $Q^*(z)\not=0$ for all these eigenvectors. The claim is proven for this case as well.
\end{enumerate}
\end{proof}
\begin{corollary}\label{realgrem}
Concerning the real case, let $Q$ be a homogeneous polynomial map from $\mathbb R^n$ to $\mathbb R^n$, of degree $m>1$, such that its complexification admits only finitely many pairwise linearly independent eigenvectors. If $m$ is even or if both $m$ and $n$ are odd then $Q$ admits a real eigenvector.
\end{corollary}
\begin{proof}
Non-real eigenvectors come in pairs of complex conjugates and $\sum_{k=0}^{n-1}m^k$ is odd with the given hypotheses. 
\end{proof}
We will improve this result in the next section.
\section{Eigenvectors: Analytic and topological methods}
\subsection{The Brouwer degree for polynomials}\label{brouwer} There exists an analytic approach to the Brouwer degree of maps on open subsets of $\mathbb R^n$, which is outlined in a comprehensive and concise manner in Deimling \cite{Deimling}, Ch.~I; see also Milnor \cite{Milnor}. We will specialize the approach to polynomial maps, following \cite{PumWa}. In the present subsection we let
\begin{equation}\label{polydef}
P:\, \mathbb R^n\to\mathbb R^n, \quad P=P_0+\cdots+ P_m
\end{equation}
be a polynomial map, with each $P_k$ homogeneous of degree $k$, and $P_m\not=0$.
\begin{enumerate}[1.]
\item We define 
\[
N_P:=\left\{ z\in \mathbb R^n; \,\det DP(z)=0\right\}.
\]
Then $P(N_P)$ is contained in a proper algebraic subset of $\mathbb R^n$ (see Shafarevich \cite{Shafa}, Ch.~II, \S1).
\item Let $\Omega\subset\mathbb R^n$ be nonempty, open and bounded. For every $y\not\in P(N_P)\cup P(\partial\Omega)$ the equation $P(x)=y$ has finitely many solutions in $\Omega$ (see \cite{Deimling}, Ch.~I, Prop. 1.3ff.), hence 
\[
d(P,\Omega,y):= \sum_{z\in P^{-1}({y})}{\rm sgn}\det DP(z)
\]
is a well-defined integer, which is called the {\em topological degree} of $P$ on $\Omega$ with respect to $y$. By a density argument and using an integral representation of the degree this definition can be extended to any $y\not\in P(\partial\Omega)$.
\item If $N_P=\mathbb R^n$ (in other words, $\det DP(x)=0$ for all $x$) then $d(P,\Omega,y)=0$ for all $y\not\in P(\partial \Omega)$.
\item If $d(P,\Omega,y)\not=0$ then $P(x)=y$ has a solution in $\Omega$.
\item Consequences of homotopy invariance (\cite{Deimling}, Ch.~I, Thm.~3.1):\\
(i)  If $R$ is a polynomial map and $y\in\mathbb R^n$  such that $y\not\in ((1-t)R+ tP)(\partial\Omega)$ for $0\leq t\leq 1$ then $d(R,\Omega,y)= d(P,\Omega,y)$.\\
(ii) If $[0,\,1]\to \mathbb R^n,\,t\mapsto \gamma(t)$ is continuous and $\gamma(t)\not\in P(\partial\Omega)$ for all $t$ then $d(P,\Omega,\gamma(0))=d(P,\Omega,\gamma(1))$.
\item Now assume that $P_m(z)\not=0$ for all $z\not=0$, which implies the existence of some $\rho>0$ such that $\Vert P_m(z)\Vert\geq \rho \Vert z\Vert^m$ for all $z$. (Take $\rho$ as the minimum of $\Vert P_m(z)\Vert$ on the unit sphere.) Then there exists an $r_0>0$ such that 
\[
d(P, B_r(0),y) =d(P_m,B_r(0),y) \text{  for all  }r\geq r_0.
\]
(Here $B_r(0)$ denotes the ball with radius $r$  and center $0$.) This follows from 5.(i) and from $\Vert P(x)-P_m(x)\Vert/\Vert P_m(x)\Vert\to 0$ as $\Vert x\Vert\to\infty$.
\item We keep the assumption that $P_m(z)\not=0$ whenever $z\not=0$. By 5.(ii) we see that for every $y$ 
\[
d(P_m,B_r(0),y)= d(P_m,B_r(0),0)=:d(P_m)
\]
for all sufficiently large $r$. Thus one may define a global topological degree for all polynomials $P$ which satisfy $P_m(z)\not=0$ whenever $z\not=0$ via
\[
d(P):=d(P_m).
\]
\end{enumerate}
We collect some pertinent facts in
\begin{lemma}\label{toplem} Let $P$ be a polynomial map such that $P_m(z)\not=0$ whenever $z\not=0$.
\begin{enumerate}[(i)]
\item Whenever $d(P)\not=0$ then every equation $P(x)=y$ has a solution in $\mathbb R^n$, and for any $y\not\in P(N_P)$ the number of solutions is at least equal to $|d(P)|$.
\item For odd $m$ every equation $P(x)=y$ has a solution in $\mathbb R^n$.
\item For even $m$ the degree $d(P)$ is even. Hence the number of solutions of $P(x)=y$ is even for every $y\not\in P(N_P)$.
\end{enumerate}
\end{lemma}
\begin{proof}Part (i) is immediate from the above. To prove part (ii), note that $d(P)$ is odd by Borsuk's theorem for $P_m$ (see \cite{Deimling}, Ch.~I, Thm.~4.1). For part (iii) note that whenever $y\not\in P_m(N_{P_m})$ then $P_m(z)=y$ if and only if $P_m(-z)=y$.
\end{proof}
We turn to the complex setting. A polynomial map
\[
S=S_0+\cdots + S_m:\,\mathbb C^n\to \mathbb C^n, \quad S_m\not=0,
\] 
may be regarded as a polynomial map $ P=P_{\mathbb R}(S)$ from $\mathbb R^{2n}$ to $\mathbb R^{2n}$, of degree $m$. Clearly $S_m(w)\not=0$ for all $w\in \mathbb C^n\setminus\{0\}$ is equivalent to $P_m(z)\not=0$ for all $z\in\mathbb R^{2n}\setminus\{0\}$. It is known that $\det DP(z)\geq 0$ for all $z$ (see a proof in the Appendix).
\begin{proposition}\label{algtop} Let $S: \mathbb C^n\to\mathbb C^n$ be a polynomial map of degree $m$ such that $S_m(w)=0$ only when $w=0$, and $P=P_{\mathbb R}(S)$. Then $d( P)=m^n$, hence the topological degree of $P$ and the number of zeros of $S$ according to Bezout's theorem are equal.

\end{proposition}
\begin{proof} $Z:= P_m(N_{ P_m})$ is contained in a proper subvariety of $\mathbb R^{2m}$. For any $a\not\in Z$ the number of solutions of $P(z)=a$ is equal to $d(P_m)$ by the previous subsection, and equal to $m^n$ by Bezout (note that all multiplicities of zeros of $P_m(x)-y$ equal $1$ whenever $y\not\in Z$).
\end{proof}
\subsection{Application to eigenvectors.}
For $\mathbb K=\mathbb C$ the analytic methods give no improvement of Proposition \ref{bezcor} and Corollary \ref{realgrem}, but for the real case they do. The first result for $\mathbb K=\mathbb R$, as well as its proof, goes back to Kaplan and Yorke \cite{KaYo}.
\begin{proposition}\label{evendegree}Let $Q:\mathbb R^n\to \mathbb R^n$ be a homogeneous polynomial map of even degree. Then there exists $v\not=0$ such that $Q(v)=0$ or $c\not=0$ such that $Q(c)=c$.

\end{proposition}
\begin{proof} Assume the contrary and define
\[
H(t,x):=(1-t)Q(x)+tx,\quad 0\leq t\leq 1.
\]
Then $H(t,x)\not=0$ for all $t$ and all $x\not=0$, since otherwise $Q(z)=\frac{t}{1-t}z$ for some $t\in[0,1)$; contradiction. Therefore 
\[
1= d({\rm id},\,B_r(0),0)=d(Q,B_r(0),0)
\]
for all $r>0$ by \#5 in the previous subsection. On the other hand, $d(Q)$ is even by Lemma \ref{toplem}; a contradiction.
\end{proof}
This result holds even when the number of one dimensional complex eigenspaces of $Q$ is infinite, as does the next result, which is a direct application of the ``hedgehog theorem'' (see Deimling \cite{Deimling}, Ch.~I, Thm.~3.4).
\begin{proposition}Let $n$ be odd and $Q:\,\mathbb R^n\to\mathbb R^n$ homogeneous. Then there exists $c\not=0$ such that $Q(c)\in\mathbb R c$.

\end{proposition}
\subsection{Gradient fields} We turn to real gradient fields, recalling the notions from subsection \ref{gradientstuff}. For these, a simple but crucial fact guarantees the existence of real eigenvectors (compare a standard existence proof for real eigenvectors of real symmetric matrices).
\begin{proposition}\label{lagrange}
Let $\mathbb K=\mathbb R$ and let $Q$ be homogeneous of degree $m$ such that $Q(x)$ is symmetric for every $x$. Then there exist $0\not=v\in \mathbb R^n$ and $\lambda\in\mathbb R$ such that $Q(v)=\lambda v$.
\end{proposition}
\begin{proof} Define $q(x):=\left<Q(x),x\right>$ and recall $Dq(x)y=(m+1)\left<Q(x),y\right>$ from \eqref{graddef} and \eqref{symmder}. On the compact sphere $\mathbb S^{(n-1}$ the function $q$ attains a maximum, say at $v$. By the Lagrange multiplier theorem one has $Q(v)=\mu v$ for some $\mu$.
\end{proof}
More detailed information about eigenvectors of gradients follows from the Poincar\'e-Hopf Theorem (see Milnor \cite{Milnor}, \S6). We state its specialization for the sphere. 
\begin{theorem}\label{poinhopf}
For a smooth vector field $ F$ on the $(n-1)$-sphere $\mathbb S^{n-1}$ the sum of the indices of its critical points is equal to the Euler characteristic of $\mathbb S^{n-1}$, thus is equal to $2$ for odd $n$ and is equal to $0$ for even $n$.
\end{theorem}
Some notions and facts should be supplemented here.
\begin{itemize}
\item A vector field on $\mathbb S^{n-1}$ is by definition tangent to $\mathbb S^{n-1}$. It may be defined locally via a parameterization of the sphere, or by restricting a suitable vector field that is defined in some neighborhood of $\mathbb S^{n-1}$.
\item The index of a critical point $a$ is generally defined in Milnor \cite{Milnor}, \S6 (via the Brouwer degree); when $a$ is nondegenerate (i.e.,  the derivative at $a$ is invertible) it is equal to the sign of its Jacobian determinant (Milnor \cite{Milnor} \S6, Lemma 4).
\end{itemize}
We now collect some pertinent facts for gradient fields. Recall that a real symmetric matrix is diagonalizable over $\mathbb R$.
\begin{proposition}\label{gradientprop}
 Let $Q:\,\mathbb R^n\to\mathbb R^n$ be homogeneous of degree $m$ such that $DQ(x)$ is symmetric for every $x$. Then:
\begin{enumerate}[(a)]
\item 
\[
Q^*:\,\mathbb R^n\setminus\{0\}\to\mathbb R^n,\quad Q^*(x):=Q(x)-\frac{\left<Q(x),\,x\right>}{\left<x,\,x\right>}x
\]
defines a vector field on $\mathbb S^{n-1}$. A point $c\in\mathbb S^{n-1}$ is stationary for $Q^*$ if and only if $Q(c)\in\mathbb R c$.
\item Let $c\in\mathbb S^{n-1}$ such that $Q(c)=\alpha c$ (hence $c$ is an eigenvector of  $DQ(c)$ with eigenvalue $m\alpha $), and denote the remaining eigenvalues of $DQ(c)$ by $\beta_2,\ldots,\beta_n$ (counted with multiplicities). Then the derivative of $Q^*|_{\mathbb S^{n-1}}$ at $c$ has eigenvalues $\beta_2-\alpha,\ldots,\beta_n-\alpha$, and the index of the stationary point is equal to the sign of $\prod(\beta_k-\alpha)$, provided this product is nonzero.
\item Let $m$ be odd. For a nondegenerate stationary point $c$ of $Q^*$, $-c$ is also a nondegenerate stationary point, and the indices of both points are equal. The same holds in case that $m$ is even and $n$ is odd.
\item Let $m$ and $n$ be even. For a nondegenerate stationary point $c$ of $Q^*$ one has that $-c$ is also a nondegenerate stationary point, and the indices of both points are different.
\end{enumerate}
\end{proposition}
\begin{proof}
Since $\left<Q^*(x),x\right>=0$ for all $x\not=0$, $Q^*$ actually defines a vector field on the sphere. A straightforward computation yields
\[
DQ^*(x)y=DQ(x)y-\frac{\left<Q(x),\,x\right>}{\left<x,\,x\right>}\cdot y+\left(\cdots\right)\cdot x
\]
hence $\left<c,c\right>=1$ and $Q(c)=\alpha c$, $DQ(c)c=m\alpha \,c$ imply
\[
DQ^*(c)y=DQ(c)y-\alpha y +\left(\cdots\right)\cdot c.
\]
Therefore 
\[
DQ(c)v=\beta v\quad \Rightarrow \quad DQ^*(c)v=(\beta-\alpha)v+\left(\cdots\right)\,c,
\]
which shows the assertion about the eigenvalues. Since the determinant is the product of the eigenvalues, and all eigenvalues are real, the remaining assertions follow by inspection.
\end{proof}
\begin{remark} From these arguments one also sees: The multiplicity of the zero $c$  of $x\mapsto Q(x)-\alpha x$ (with derivative $DQ(c)-\alpha I$) is equal to one if and only if the corresponding stationary point of $Q^*$ is nondegenerate.
\end{remark}

\section{Applications to polynomial differential equations}
Obviously the results of subsection \ref{brouwer} may be employed to show the existence of stationary points for polynomial differential equations, and to obtain bounds on their number. There are two applications which refer specifically to eigenvectors, and we discuss these in some detail.
\subsection{Particular solutions of homogeneous polynomial systems.} Here we follow R\"ohrl \cite{RohrlId}. Given a homogeneous polynomial map $Q:\, \mathbb K^n\to\mathbb K^n$ of degree $m>1$, consider the differential equation $\dot x=Q(x)$. Let $0\not=c\in\mathbb K^m$ such that $Q(c)=\alpha c$. If $\alpha=0$ then every point on the line $\mathbb K c$ is stationary. For $\alpha\not=0$ make the ansatz $v(t)=\phi(t)\cdot c$ for a solution to obtain the necessary and sufficient condition
\[
\dot \phi(t)\cdot c=Q(\phi(t)\cdot c)=\alpha\phi(t)^m\cdot c,
\]
which yields the separable one dimensional equation $\dot y=\alpha y^m$; this admits solution by elementary functions.\\
Thus we have verified the existence and the explicit form of certain particular solutions. But in contrast to the linear case one should not expect to obtain further solutions as explicit ``combinations'' (of whatever type) of those special ones, and neither should one expect to obtain real solutions from solutions related to non-real eigenspaces. Moreover, no higher dimensional invariant subspaces exist in general for homogeneous $Q$ of degree $>1$; see \cite{WADE}, Ch.~3-4 for a discussion.
\subsection{Stationary points at infinity} We consider a polynomial differential equation
\begin{equation}\label{polyde}
\dot x=P(x)=P_0(x)+\cdots+P_m(x), \quad m>1,\, P_m\not=0\quad \text{  on  }\mathbb R^n.
\end{equation}
One may associate to this equation a system on the $n$-sphere $\mathbb S^n\subseteq\mathbb R^{n+1}$ (the {\em Poincar\'e hypersphere}) with a $\mathbb Z_2$ symmetry (or a system on the projective space $\mathbb P^n$, if one prefers). We describe the construction, which is mostly used in dimension $n=2$ but generally valid, following Perko \cite{Perko}, Section 3.10 for geometric motivation, and following \cite{RW}, Section 2 for a general (somewhat abstract) approach. 
\begin{itemize}
\item The construction, geometrically (see Perko \cite{Perko} for illustrations in dimension two): Consider $\mathbb S^n\subseteq\mathbb R^{n+1}$ and identify the phase space $\mathbb R^n$ with the tangent hyperplane $\mathbb H$ to the sphere at the north pole $(0,\ldots,0,1)^{\rm tr}$. The central projection from $0$ then induces a bijection from the upper hemisphere $\{x\in\mathbb S^n; \, x_{n+1}>0\}$ to $\mathbb H$, and the equator $\{x\in\mathbb S^n; \, x_{n+1}=0\}$ may be identified with ``points at infinity'' of $\mathbb H$. The image of the  polynomial vector field on $\mathbb H$ can be extended to a rational vector field on $\mathbb S^n$ and there remains, up to orbital equivalence, a polynomial vector field.
\item Antipodal pairs of stationary points of this vector field on the equator (``at infinity'') stand in 1-1 correspondence with one dimensional real eigenspaces of the highest degree term $P_m$. If $P_m(c)=\alpha c$ then $m\alpha$ is an eigenvalue of $DP_m(c)$ with eigenvector $c$. Let $\beta_2,\ldots, \beta_n$ be the remaining (possibly complex) eigenvalues of $DP(c)$. Then the linearization of the vector field on the sphere at the stationary point has eigenvalues $-\alpha$ and $\beta_2-\alpha,\ldots,\beta_n-\alpha$. The sum of the generalized eigenspaces for the latter is just the tangent space to the equator, and there exists an eigenvector for $-\alpha$ which is transversal to this tangent space. (See \cite{RW}, Lemma 2.1 for details and a proof. Note that the eigenvalues $\beta_k-\alpha$ also appear in Proposition \ref{gradientprop}; this is not a coincidence.)
\end{itemize}
From this approach we directly obtain criteria for unbounded solutions to polynomial differential systems. The statement and a proof of part (b) for quadratic systems is due to Kaplan and Yorke; for part (a) see \cite{RW}, Prop.~2.2. 
\begin{proposition} Let the ordinary differential equation \eqref{polyde} be given.
\begin{enumerate}[(a)]
\item If there exists a nonzero $c\in\mathbb R^n$ and an $\alpha >0$ such that $P_m(c)=\alpha c$ then the system has an unbounded solution.
\item If $m$ is even then the system has an unbounded solution or every one dimensional eigenspace of $P_m$ is spanned by a nilpotent.
\end{enumerate}
\end{proposition}
\begin{proof}[Sketch of proof]
For the first part consider the stationary point at infinity which corresponds to the positive half-line $\mathbb R_+c$. Due to $-\alpha<0$, the stable manifold of this stationary point has nontrivial intersection with the upper hemisphere. For part (b) use Proposition \ref{evendegree}.
\end{proof}
This straightforward result uses only part of the available information. More detailed qualitative studies are possible, in particular in dimension two. We refer to Perko \cite{Perko}, Sections 3.10-3.11, Dumortier et al.~\cite{DLA}, and also \cite{RW}, Sections 3 and 4. 
\section{Extrema of homogeneous cubic polynomials}
\subsection{Background} Motivated by the structure theory of  liquid crystals, Gaeta and Virga \cite{GaVi} discussed harmonic homogeneous cubic polynomials $q$ on $\mathbb R^3$. Concerning the physics we refer to \cite{GaVi}, Virga \cite{Virga} and the references given there.
Our principal interest lies in the distribution of critical points of $q$ on the sphere $\mathbb S^2$. This question was discussed, using various methods, by Gaeta and Virga \cite{GaVi}, and by Chen, Qi and Virga \cite{CQV}. An interesting observation in \cite{GaVi} concerns the number of maxima of $q$ on the sphere: Theorem \ref{poinhopf} and Proposition \ref{gradientprop} would permit any number of maxima between one and four in nondegenerate settings, but actually only cases with three or four maxima could be observed. The condition $\Delta q=0$ seems to be crucial for this phenomenon, since there is no hope that general real cubic gradient systems should have more than one maximum:
\begin{example}
The polynomial $q(x)=(x_1^2+x_2^2+x_3^2)\cdot (x_1+x_2+x_3)$ admits only one maximum and one minimum on $\mathbb S^2$.
\end{example}
We discuss the class of harmonic cubic homogeneous polynomials in three variables with a somewhat different approach, reducing their investigation to determining the number of real zeros of a real degree six polynomial in one variable. For the latter task, algorithms are available.
\subsection{The setting}
The first task (here as in the cited references) is to reduce the number of parameters. With \eqref{graddef} we have a homogeneous quadratic map
\[
Q:\,\mathbb R^3\to\mathbb R^3, \quad DQ(x) \text{ symmetric,  } \quad{\rm tr}\,DQ(x)=0 \text{  for all  }x
\]
such that the identity $q(x)=\frac13\left<Q(x),x\right>$ holds.
For any $x\in\mathbb R^3$ we define the linear map $L(x)$ by
\[
L(x)y:= \widetilde Q(x,y):=\frac12 D^2Q(0)(x,y),
\]
thus $DQ(x)=2L(x)$ and $L(x)y=L(y)x$ for all $x$ and $y$.
Upon fixing a basis $c_1,\,c_2,\, c_3$ of $\mathbb R^3$, $Q$ is determined by the $\widetilde Q(c_i,c_j)$, due to bilinearity. Similar to Gaeta and Virga \cite{GaVi} we choose the basis in a convenient way but take a different path. First, there exists $c_1$ of norm one such that $q(c_1)$ is minimal on the sphere, and hence $Q(c_1)=L(c_1)c_1=\alpha_1 c_1$ for some $\alpha$. (In contrast to \cite{GaVi} and \cite{CQV} we consider a minimum rather than a maximum, and place it at the east pole of the sphere rather than the north pole. We also eliminate parameters by a different strategy.) Since $L(c_1)$ is symmetric, $c_1$ can be extended to an orthonormal basis by $c_2,\,c_3$ such that 
\[
L(c_1)c_2=\alpha_2c_2,\quad L(c_1)c_3=\alpha_3c_3
\]
for suitable real $\alpha_2,\,\alpha_3$, and $\alpha_1=-(\alpha_2+\alpha_3)$ due to trace zero. From the ansatz
\[
\begin{array}{rcl}
\widetilde Q(c_2,c_2)&=& \beta_1c_1+\beta_2c_2+\beta_3c_3\\
\widetilde Q(c_2,c_3)&=& \gamma_1c_1+\gamma_2c_2+\gamma_3c_3\\
\widetilde Q(c_3,c_3)&=& \delta_1c_1+\delta_2c_2+\delta_3c_3\\
\end{array}
\]
with real coefficients $\beta_i,\gamma_i,\delta_i$ and using symmetry of the tensor $\widetilde Q$, we obtain matrix representations
\[
L(c_2)=\begin{pmatrix}0&\alpha_2&0\\
                      \beta_1&\beta_2&\beta_3\\
                     \gamma_1&\gamma_2&\gamma_3
\end{pmatrix}; \quad L(c_3)=\begin{pmatrix}0&0&\alpha_3\\
                     \gamma_1&\gamma_2&\gamma_3\\
\delta_1&\delta_2&\delta_3
\end{pmatrix}
\]
with respect to the chosen basis. These matrices are symmetric with trace zero, which yields
\[
\gamma_1=0;\,\delta_1=\alpha_3;\,\gamma_2=\beta_3=-\delta_3;\, \gamma_3=\delta_2=-\beta_2.
\]
Thus only four parameters $\alpha_2,\alpha_2,\beta_2,\beta_3$ remain. We record an intermediate result:
\begin{lemma} In coordinates with respect to the basis $c_1,\,c_2,\,c_3$ one has the representation
\begin{equation}\label{symmtens}
Q(x)=\begin{pmatrix} -(\alpha_2+\alpha_3)x_1^2+\alpha_2x_2^2+\alpha_3x_3^2\\
                                      2\alpha_2x_1x_2+\beta_2x_2^2+2\beta_3x_2x_3-\beta_2x_3^2\\
                                      2\alpha_3x_1x_3+\beta_3x_2^2-2\beta_2x_2x_3-\beta_3x_3^2
\end{pmatrix}.
\end{equation}
\end{lemma}
We will from now on assume that the $c_i$ form the standard basis of $\mathbb R^3$.
\begin{remark}\label{minrem}
The requirement for $q$ to have a minimum at $c_1$ has further implications for the $\alpha_i$. By Proposition \ref{gradientprop}, the derivative of $Q^*|_{\mathbb S^2}$ has the two eigenvalues $3\alpha_2+\alpha_3$ and $\alpha_2+ 3\alpha_3$, both of which must be $\geq 0$. Moreover one finds (e.g. by considering values of $q$ along great circles through $c_1$) that in case $3\alpha_2+\alpha_3=\alpha_2+ 3\alpha_3=0$ (equivalently $\alpha_2=\alpha_3=0$) there is no minimum at $c_1$. Therefore one eigenvalue must be $>0$.
\end{remark}
To determine the one dimensional eigenspaces of $Q$ we note that every $2\times 2$ minor of the matrix
\[
M(x):=\begin{pmatrix}Q(x)&x\end{pmatrix}
\]
must be zero on such an eigenspace, and conversely any nontrivial common zero of these minors gives rise to a one dimensional eigenspace.
For the last two rows of $M$, evaluation of the determinant yields
\begin{equation}\label{minorone}
0=2(\alpha_2-\alpha_3)x_1x_2x_3+3\beta_2x_2^2x_3+3\beta_3x_2x_3^2-\beta_2x_3^3-\beta_3x_2^3,
\end{equation}
while the first two rows yield the condition
\begin{equation}\label{minortwo}
0=-(3\alpha_2+\alpha_3)x_1^2x_2+x_1\,\left(-\beta_2x_2^2-2\beta_3x_2x_3+\beta_2x_3^2\right)+\alpha_2x_2^3+\alpha_3x_2x_3^2.
\end{equation}
Multiplying \eqref{minortwo} by $x_2x_3^2$ one finds
\begin{equation}\label{minortwomod}
0=-(3\alpha_2+\alpha_3)v^2+\left(-\beta_2x_2^2x_3-2\beta_3x_2x_3^2+\beta_2x_3^3\right)\, v+\alpha_2x_2^4x_3^2+\alpha_3x_2^2x_3^4
\end{equation}
with $v=x_1x_2x_3$. Evaluating the first and last row of $M$, we find
\begin{equation}\label{minorthree}
0=(\alpha_2+3\alpha_3)x_1^2x_3+x_1\,\left(\beta_3x_2^2-2\beta_2x_2x_3-\beta_3x_3^2\right)-\alpha_2x_2^2x_3-\alpha_3x_3^3
\end{equation}
and multiplication by $x_2^2x_3$ yields
\begin{equation}\label{minorthreemod}
0=(3\alpha_3+\alpha_2)v^2+\left(\beta_3x_2^3-2\beta_2x_2^2x_3-\beta_3x_2x_3^2\right)\, v-\alpha_2x_2^4x_3^2-\alpha_3x_2^2x_3^4.
\end{equation}
Taking the difference of \eqref{minorthreemod} and \eqref{minortwomod} we obtain a symmetrized version
\begin{equation}\label{minorfourmod}
0=4(\alpha_2+\alpha_3)v^2+\left(\beta_3x_2^3-\beta_2x_2^2x_3+\beta_3x_2x_3^2-\beta_2x_3^3\right)\, v-2(\alpha_2x_2^4x_3^2+\alpha_3x_2^2x_3^4).
\end{equation}
The determinantal conditions are not independent; for instance, any common solution $(y_1,y_2,y_3)^{\rm tr}$ of \eqref{minorone} and \eqref{minortwo}  with $y_2\not=0$ automatically satisfies \eqref{minorthree}.
Our goal is now to reduce the problem of determining (real) one dimensional eigenspaces to a polynomial in one variable. We will have to dispose of some special cases first and then look at the generic case.
\subsection{Special cases} We first note from \eqref{symmtens} that there exists a one dimensional eigenspace $\mathbb R\cdot(1,0,y_3)^{\rm tr}$ with $y_3\not=0$ only if $\beta_2=0$, and mutatis mutandis this holds with labels 2 and 3 interchanged. When $\beta_2=0$ one furthermore sees that $Q$ maps the subspace given by $x_2=0$ to itself. We will investigate these cases in some detail, starting with the most special one.
\subsubsection{The case $\beta_2=\beta_3=0$} Here we have
\[
Q(x)=\begin{pmatrix} -(\alpha_2+\alpha_3)x_1^2+\alpha_2x_2^2+\alpha_3x_3^2\\
                                      2\alpha_2x_1x_2\\
                                      2\alpha_3x_1x_3
\end{pmatrix}
\]
and \eqref{minorone}, \eqref{minortwo} and \eqref{minorthree} specialize to
\[
\begin{array}{rcl}
0&=& 2(\alpha_2-\alpha_3)x_1x_2x_3\\
0&=& x_2\cdot(-(3\alpha_2+\alpha_3)x_1^2+\alpha_2x_2^2+\alpha_3x_3^2)\\
0&=& x_3\cdot((\alpha_2+3\alpha_3)x_1^2-\alpha_2x_2^2-\alpha_3x_3^2)\\
\end{array}
\]
When $\alpha_2\not=\alpha_3$ then the first condition yields three cases for one dimensional eigenspaces (except $\mathbb R c_1$).
\begin{itemize}
\item $x_1=0$ and $\alpha_2x_2^2+\alpha_3x_3^2=0$. This yields real solutions only if $\alpha_2$ and $\alpha_3$ have different signs or one of them equals zero.
\item $x_2=0$ and $(\alpha_2+3\alpha_3)x_1^2-\alpha_3x_3^2=0$. Since $3\alpha_3+\alpha_2\geq 0$, the existence of real solutions depends on the sign of $\alpha_3$.
\item $x_3=0$ and $-(3\alpha_2+\alpha_3)x_1^2+\alpha_2x_2^2=0$. Mutatis mutandis, the previous remark applies.
\end{itemize}
Disregarding the degenerate cases that $\alpha_2$, $\alpha_3$ or one of the eigenvalues vanish (which are easily analyzed), and recalling that one of the $\alpha_i$ must be positive, one always finds that precisely two of the conditions are satisfied, leading to four real solutions (and a total of five one dimensional eigenspaces).\\
This leaves the case $\alpha_2=\alpha_3=:\alpha$, with
\[
Q(x)=\begin{pmatrix} -2\alpha x_1^2+\alpha x_2^2+\alpha x_3^2\\
                                      2\alpha x_1x_2\\
                                      2\alpha x_1x_3
\end{pmatrix}.
\]
Then \eqref{minorone} through \eqref{minorthree} leave only one nontrivial condition
\[
-4x_1^2+x_2^2+x_3^2=0.
\]
Here we have infinitely many one dimensional eigenspaces. The union of these is a quadric united with $\mathbb R c_1$.
\subsubsection{The case $\beta_2=0$, $\beta_3\not=0$} Here  \eqref{minorone}, \eqref{minortwo} and \eqref{minorthree} specialize to
\[
\begin{array}{rcl}
0&=& x_2\left(2(\alpha_2-\alpha_3)x_1 x_3 -\beta_3x_2^2+3\beta_3x_3^2   \right)\\
0&=& x_2\left(-(3\alpha_2+\alpha_3)x_1^2-2\beta_3x_1x_3+\alpha_2x_2^2+\alpha_3x_3^2\right)\\
0&=& (\alpha_2+3\alpha_3)x_1^2x_3+x_1\beta_3(x_2^2-x_3^2)-\alpha_2x_2^2x_3-\alpha_3x_3^3\\
\end{array}
\]
We first determine the one dimensional eigenspaces on the plane defined by $x_2=0$. Excluding $\mathbb R c_1$, the remaining condition is
\[
 (\alpha_2+3\alpha_3)x_1^2-\beta_3 x_1x_3-\alpha_3x_3^2=0
\]
which yields at most two more real eigenspaces. (Note that $\alpha_2+3\alpha_3$ and $\alpha_3$ cannot both be zero.)\\
For one dimensional eigenspaces outside the plane with $x_2=0$ the two remaining conditions are
\[
\begin{array}{rcl}
0&=& 2(\alpha_2-\alpha_3)x_1 x_3 -\beta_3x_2^2+3\beta_3x_3^2  \\
0&=& -(3\alpha_2+\alpha_3)x_1^2-2\beta_3x_1x_3+\alpha_2x_2^2+\alpha_3x_3^2\\\
\end{array}
\]
There are two special subcases to be considered. First, one may have $\alpha_2=\alpha_3$, hence $x_2^2-3x_3^2=0$ from the first condition, and for each solution $(y_2,y_3)$ of the latter one finds two real solutions for $x_1$ in the second condition (noting that the leading and the constant coefficient have different signs). In total, we have then seven one dimensional eigenspaces.\\
Second, one may have $3\alpha_2+\alpha_3=0$, which leaves the conditions
\[
\begin{array}{rcl}
0&=& 8\alpha_2 x_1 x_3 -\beta_3(x_2^2- 3x_3^2)  \\
0&=& -2\beta_3x_1x_3+\alpha_2(x_2^2-3x_3^2)\\
\end{array}
\]
If $4\alpha_2^2-\beta_3^2\not=0$ then we have two real (and also only two complex) solutions for the system, with $x_2^2-3x_3^2=0$ and $x_1=0$, and in total there are five one -dimensional eigenspaces (real as well as complex). In case $4\alpha_2^2-\beta_3^2=0$ we see that there are infinitely many one dimensional eigenspaces; their union forms the quadric $\pm 4x_1x_3-x_2^2+3x_3^2=0$.\\
For the remaining cases it suffices to consider solutions with $x_3=1$. Setting $t:=x_2/x_3$, the first condition shows that
\[
2(\alpha_2-\alpha_2)x_1-\beta_3(t^2-3)=0,
\]
and substitution into the second condition provides the fourth-degree polynomial equation
\[
\begin{array}{rcl}
-(3\alpha_2+\alpha_3)( t^2-3)^2&-&4\left(\beta_3^2(\alpha_2-\alpha_3)-\alpha_2(\alpha_2-\alpha_3)^2\right)(t^2-3)\\
&+&4(\alpha_2-\alpha_3)^2\cdot(3\alpha_2+\alpha_3)
\end{array}
\]
The quadratic polynomial 
\[
-(3\alpha_2+\alpha_3)s^2-4\left(\beta_3^2(\alpha_2-\alpha_3)-\alpha_2(\alpha_2-\alpha_3)^2\right)+
4(\alpha_2-\alpha_3)^2\cdot(3\alpha_2+\alpha_3)
\]
admits one positive and one negative root; note the signs of the leading and the constant coefficients. Since 
$t^2-3=s^*$ has two real solutions for every $s^*>0$, we see that there are at least two real solutions for the quartic equation in $t$. Concerning tools for a more detailed analysis, see the following section.
\subsubsection{The case $\beta_3=0$, $\beta_2\not=0$} Mutatis mutandis, this is completely analogous to the case we just discussed.
\subsection{The generic case} Having disposed of special cases, we may now require the following generic conditions:
\begin{equation}\label{genericcase}
\beta_2\not=0,\quad \beta_3\not=0,\quad \alpha_2\geq \alpha_3
\end{equation}
and we know that $\alpha_2+\alpha_3>0$ since there is a minimum at $c_1$ (recall Remark \ref{minrem}).
Moreover the only eigenspace $\mathbb R (y_1,y_2,y_3)^{\rm tr}$ with $y_2=0$ or $y_3=0$ is $\mathbb R c_1$, so the remaining eigenspaces correspond to the common zeros of \eqref{minorone} and \eqref{minorfourmod}.
\subsubsection{The case $\alpha_2=\alpha_3$} We deal with this nongeneric subcase of the generic case first. Then \eqref{minorone} simplifies to the homogeneous equation
\[
3\beta_2x_2^2x_3+3\beta_3x_2x_3^2-\beta_2x_3^3-\beta_3x_2^3=0,
\]
which has three distinct real homogeneous solutions: Indeed, by setting $\gamma:=\beta_3/\beta_2$ and $t:=x_3/x_2$ homogeneous solutions of the above correspond to zeros of the polynomial
\[
\mu(t)=-t^3+3\gamma t^2-3t-\gamma.
\]
Its derivative 
\[
\mu^\prime(t)=-3(t^2-2\gamma t-1)
\]
has two real zeros, one positive and one negative. Each zero $t^*$ of the derivative satisfies ${t^*}^2=2\gamma t^*+1$ and furthermore
${t^*}^3=2\gamma {t^*}^2+t^*$. Altogether this yields
\[
\mu(t^*)=(2\gamma^2+2)t^*,
\]
therefore $\mu$ admits a negative local minimum and a positive local maximum, hence three real zeros. Each corresponding homogeneous zero $(x_2,x_3)$ yields a quadratic equation \eqref{minorfourmod} for $v$ positive leading coefficient and negative constant coefficient, hence there are two real solutions for $v$. Altogether we have six common homogeneous zeros of \eqref{minorone}, \eqref{minorfourmod}, and a total of seven one dimensional eigenspaces.
\subsubsection{The case $\alpha_2\not=\alpha_3$} We may then furthermore require that $2(\alpha_2-\alpha_3)=1$, since multiplying each coefficient $\alpha_2,\alpha_3, \beta_2,\beta_3$ by a positive number leaves solutions of \eqref{minorone} and \eqref{minorfourmod} unchanged, and the minimum at $c_1$ remains a minimum. We thus have $\alpha_3=\alpha_2-\frac12$, and the positivity condition on the eigenvalues is equivalent to $\alpha_2>\frac38$. Setting $t:=x_3/x_2$, equation \eqref{minorone} implies that 
\[
v=\beta_2t^3-3\beta_3t^2-3\beta_2t+\beta_3,
\]
and substitution into \eqref{minorfourmod} yields a degree six polynomial
\[
\rho(t)=\sum_{j=0}^6\gamma_jt^j
\]
whose zeros correspond to the one dimensional eigenspaces of $Q$. A straightforward computation yields
\begin{equation}
\begin{array}{rcl}
\gamma_6&=& \beta_2^2(8\alpha_2-3);\\
\gamma_5&=&16\beta_2\beta_3(-3\alpha_2+1);\\
\gamma_4&=& -48 \alpha_2\beta_2^2+72\alpha_2\beta_3^2+14\beta_2^2-21\beta_3^2-2\alpha_2+1;\\
\gamma_3&=&40\beta_2\beta_3(4\alpha_2-1);\\
\gamma_2&=&72\alpha_2\beta_2^2-48\alpha_2\beta_3^2-15\beta_2^2+10\beta_3^2-2\alpha_2;\\
\gamma_1&=&8\beta_2\beta_3(-6\alpha_2+1);\\
\gamma_0&=& \beta_3^2(8\alpha_2-1).
\end{array}
\end{equation}
This may look a bit unwieldy, but note that all the information about real eigenspaces of $Q$ is now contained in the one-variable polynomial $\rho$ with parameters that satisfy $\alpha_2\geq\frac38$, $\beta_2\not=0\not=\beta_3$  but are otherwise arbitrary. Computing a Sturm sequence for $\rho$ (essentially computing the gcd of $\rho$ and $\rho^\prime$ by Euclid's algorithm, see Gantmacher \cite{Gant}, Ch.~V for details) and counting sign changes yields an algorithm to determine the number of real zeros. Moreover, parameter combinations for which the gcd is not constant will indicate a transition in the number of real zeros. We will not carry out the program here (which will likely require some computing power and specialized algorithms), but are satisfied with with the observation that an algorithmic path is now open for determining the number of real eigenspaces.
\section{Appendix}
\subsection{Some linear algebra}
For the sake of completeness we prove the following commonly known fact:
For real $n\times n$ matrices $A$ and $B$ one has
\[
\det \begin{pmatrix} A&-B\\
                                   B&A\end{pmatrix}\geq 0.
\]
First assume that $A$ is invertible. Then
\[
\begin{pmatrix} A&-B\\
                                   B&A\end{pmatrix}=\begin{pmatrix}A & 0\\ 0&A\end{pmatrix}\cdot\begin{pmatrix}I_n &0\\ C & I_n\end{pmatrix}\cdot\begin{pmatrix}I_n &- C\\ 0 & I_n+C^2\end{pmatrix}
\]
where $I_n$ denotes the $n\times n$ identity matrix and $C=A^{-1}B$. The determinant of the first factor is equal to $(\det A)^2$, hence $>0$; the second factor has determinant one, and for the third one we get
\[
\det(I_n+C^2)=\det(I_n+ iC)\cdot \overline{\det(I_n+iC)}>0.
\]
In the case of non-invertible $A$ apply the argument to $\rho I_n+A$ for small $\rho>0$ and use continuity as $\rho \to 0$.\\

To apply this to real representations of complex linear maps, let $w=x+iy\in\mathbb C^n$ with $x,\,y\in\mathbb R^n$, and write a linear map $L$ from $\mathbb C^n$ to $\mathbb C^n$ in the form
\[
L= A+iB; \quad L(x+iy)= \left(Ax-By\right) + i\left(Bx+Ay\right);
\]
hence $L$ has a real matrix representation
\[
\begin{pmatrix}A&-B\\ B&A\end{pmatrix}.
\]
\subsection{A hybrid proof of Bezout's theorem}
This proof uses Proposition \ref{resultant} as an algebraic tool, and properties of the Brouwer degree on the analytic side. It is not central to the topic of the present paper but its inclusion is unproblematic and it may be seen as informative anyway. \\
By Proposition \ref{resultant}, in the complex affine space of structure coefficients, the coefficient sets which correspond to homogeneous polynomial maps $Q$ with a nilpotent ($v\not=0$ and $Q(v)=0$) form the hypersurface defined by the resultant. This hypersurface has real codimension two, hence any two points in its complement can be connected by a continuous curve. These points correspond to two homogeneous polynomial maps without nilpotents. By homotopy invariance they have the same Brouwer degree. To summarize, any two homogeneous polynomial maps of degree $m$ without nilpotents have the same degree. For the special map
\[
\widetilde Q:\,\mathbb C^n\to\mathbb C^n,\quad x\mapsto\begin{pmatrix}x_1^m\\ \vdots \\x_n^m\end{pmatrix}
\]
one easily verifies that the number of solutions of $\widetilde Q(x)=(1,\ldots,1)^{\rm tr}$ is equal to $m^n$, with each solution having multiplicity one.


\begin{thebibliography}{99}
\bibitem{CaSt} D.~Cartwright, B.~Sturmfels: {\it The number of eigenvalues of a tensor.} Linear Alg. Appl. {\bf 438}, 942--952 (2013).
\bibitem{CQV} Y.~Chen, L.~Qi, E.G.~Virga: {\it Octupolar tensors for liquid crystals.} J. Phys. A {\bf 51}, no. 2, 025206 (2018).
\bibitem{CLSuse} D.A.~Cox, J.~Little, D.~O'Shea: {\it Using algebraic geometry.} Springer, New York (2004).
\bibitem{DeLo} W.~Decker, Ch.~ Lossen: {\it Computing in algebraic geometry}.
{Algorithms and computation in mathematics} {\bf 16}, Springer, Berlin (2006).
\bibitem{Deimling} K.~Deimling: {\it Nonlinear functional analysis.} Springer, Berlin (1985).
\bibitem{DLA} F.~Dumortier, J.~Llibre, J.~Artes: {\it Qualitative theory of planar differential systems.} Springer, Berlin (2006).
\bibitem{GaVi} G.~Gaeta, E.G.~Virga: {\it Octupolar order in three dimensions.} Eur. Phys. J. E{\bf 39}, 113 (2016).
\bibitem{Gant} F.R.~Gantmacher: {\it Applications of the theory of matrices.} Dover,  Mineola (2005).
\bibitem{KaYo} J.L.~Kaplan, J.A.~Yorke: {\it Nonassociative, real algebras and quadratic differential equations.} Nonlinear Analysis {\bf 3}, 49--51 (1977).
\bibitem{Milnor} J.W.~Milnor: {\it Topology from the differentiable viewpoint.} Princeton University Press, Princeton (1997).
\bibitem{Perko} L.~Perko: {\it Differential equations and dynamical systems.} Springer, New York (1991).
\bibitem{PumWa} S.~Pumpl\"un, S.~Walcher: {\it On the zeros of polynomials over quaternions.} Comm. Algebra {\bf30}, 4007--4018 (2002).
\bibitem{OeRoStu} L.~Oeding, E.~Robeva, B.~Sturmfels: {\it Decomposing tensors into frames.} Advances in Applied Mathematics {\bf 73}, 125-153 (2016).
\bibitem{Qi05} L.~Qi: {\it Eigenvalues of a supersymmetric tensor.} J. Symb. Comp. {\bf 40}, 1302--1324 (2005).
\bibitem{Qi07} L.~Qi: {\it Eigenvalues and invariants of tensors.} J. Math. Anal. Appl. {\bf 325}, 1363--1377 (2007).
\bibitem{RohrlId} H.~R\"ohrl: {\it A theorem on nonassociative algebras and its application to differential equations.} Manuscr. Math. {\bf 21}, 181--187 (1977).
\bibitem{RohrlZ} H.~R\"ohrl: {\it On the zeros of polynomials over arbitrary finite dimensional algebras.}  Manuscr. Math. {\bf 25}, 359--390 (1978).
\bibitem{RW} H.~R\"ohrl, S.~Walcher: {\it Projections of polynomial vector fields and te Poincar\'e sphere.} J. Differential Eq. {\bf 139}, 22--40 (1997).
\bibitem{Shafa} I.R.~Shafarevich: {\it Basic algebraic geometry.} Springer, Berlin (1977).
\bibitem{Virga} E.G.~Virga: {\it Octupolar order in two dimensions.} Eur. Phys. J. E{\bf 38}, 63 (2015).
\bibitem{WADE} S.~Walcher: {\it Algebras and differential equations.} Hadronic Press, Palm Harbor (1991).
\end{thebibliography}
\end{document}